\documentclass[12pt]{article}%
\usepackage{amsmath}
\usepackage{amsfonts}
\usepackage{amssymb}
\usepackage{graphicx}%
\providecommand{\U}[1]{\protect\rule{.1in}{.1in}}
\newtheorem{theorem}{Theorem}

\newtheorem{corollary}[theorem]{Corollary}

\newtheorem{definition}[theorem]{Definition}
\newtheorem{example}[theorem]{Example}

\newtheorem{lemma}[theorem]{Lemma}

\newtheorem{proposition}[theorem]{Proposition}
\newtheorem{remark}[theorem]{Remark}

\newenvironment{proof}[1][Proof]{\noindent\textbf{#1.} }{\ \rule{0.5em}{0.5em}}
\begin{document}

\title{An Arithmetic Metric}
\author{Diego Dominici \thanks{e-mail: dominicd@newpaltz.edu}\\Department of Mathematics\\State University of New York at New Paltz\\1 Hawk Dr.\\New Paltz, NY 12561-2443\\USA\\Phone: (845) 257-2607\\Fax: (845) 257-3571 }
\maketitle

\begin{abstract}
What is the distance between 11 (a prime number) and 12 (a highly composite number)? If your answer is 1, then ask yourself "is this reasonable?"
In this work, we will introduce a distance between natural numbers based on their arithmetic properties, instead of their position on the real line. 
\end{abstract}

Keywords: Distance, prime factors, p-adic valuation.

MSC-class: 54E35 (Primary) 11A05 , 11E95, 11A25 (Secondary).

\section{Introduction}

When the concepts of distance and metric space are introduced in a standard
advanced calculus course, it is customary to present some examples of metrics.
These usually consist of the absolute value (for $\mathbb{R}$), the $l_{p}$
norms (for $\mathbb{R}^{n}$ and $\mathbb{R}^{\mathbb{N}}$) and the $L_{p}$
norms (for $\mathbb{R}^{\mathbb{R}})$ \cite{MR0085462}. In most courses, the
only "exotic" metric that students learn about is the discrete metric%
\[
d(x,y)=\left\{
\begin{array}
[c]{c}%
0\quad\text{if \ }x=y\\
1\quad\text{if \ }x\neq y
\end{array}
\right\vert .
\]

For those students, whose main interest is algebra, these examples seem to
imply that the theory of metric spaces is something that they should not care
about (except for the brief moment when they need to pass the required course!).

The objective of this article is to provide a non-trivial example of a metric
that should be interesting to algebraists and analysts alike. It should also
appeal to those interested in graph theory and discrete mathematics.

To motivate our definition, let's consider the following question: What is the
distance between $11$ and $12?$ As real numbers, the answer is of course
$d(11,12)=\left\vert 12-11\right\vert =1.$ However, if we take into account
their arithmetic properties, they are very different numbers indeed. While
$11$ is a prime number, $12$ is a highly composite number, i.e., it has more
divisors than any smaller natural number. Thus, it seems that the distance
between them as natural numbers should be based on divisibility rather than on
their location on the real line.

We can look at the problem from a slightly different perspective, if we
consider the \emph{Hasse diagram} of the set $I_{12}=\left\{  1,2,\ldots
,12\right\}  ,$ i.e., the graph formed with numbers $1,2,\ldots,12$ as
vertices and edges connecting two numbers $a<b$ iff $a|b$ (see Figure 1). If
we define the distance between two numbers in the Hasse diagram as the number
of edges in a shortest path connecting them, then clearly we have
$d(11,12)=4.$ This result seems more satisfactory than the previous
calculation using the absolute value.

\begin{figure}[ptb]
\begin{center}
\resizebox{6cm}{!}{\includegraphics{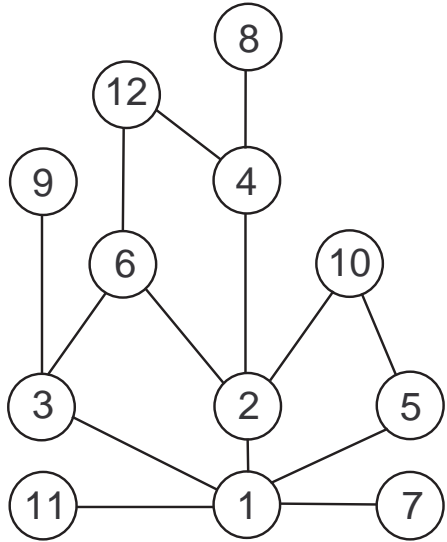}}
\end{center}
\caption{The Hasse diagram of the set $I_{12}$.}%
\label{1}%
\end{figure}

If we carefully examine the Hasse diagram, we conclude that our proposed
distance should have the following properties:

\begin{enumerate}
\item If $a<b,$ then%
\begin{equation}
d(a,b)=1\Leftrightarrow\exists p\in\mathbb{P}\text{ \ such that }%
b=ap,\label{Prop1}%
\end{equation}
where%
\[
\mathbb{P}=\left\{  p\in\mathbb{N\quad}|\quad p\text{ \ is a prime
number}\right\}  .
\]
In other words, the only way of advancing from one number to another $1$ unit
of distance is by multiplying the number by a prime.

\item If $l=\operatorname{lcm}(a,b)$ and $g=\gcd(a,b),$ then%
\begin{equation}
d(a,l)+d(l,b)=d(a,g)+d(g,b),\label{Prop2}%
\end{equation}
which says that the distance between $a$ and $b$ going through
$\operatorname{lcm}(a,b)$ is the same as going through $\gcd(a,b)$ (see Figure 2).
\end{enumerate}

\begin{figure}[ptb]
\begin{center}
\resizebox{6cm}{!}{\includegraphics{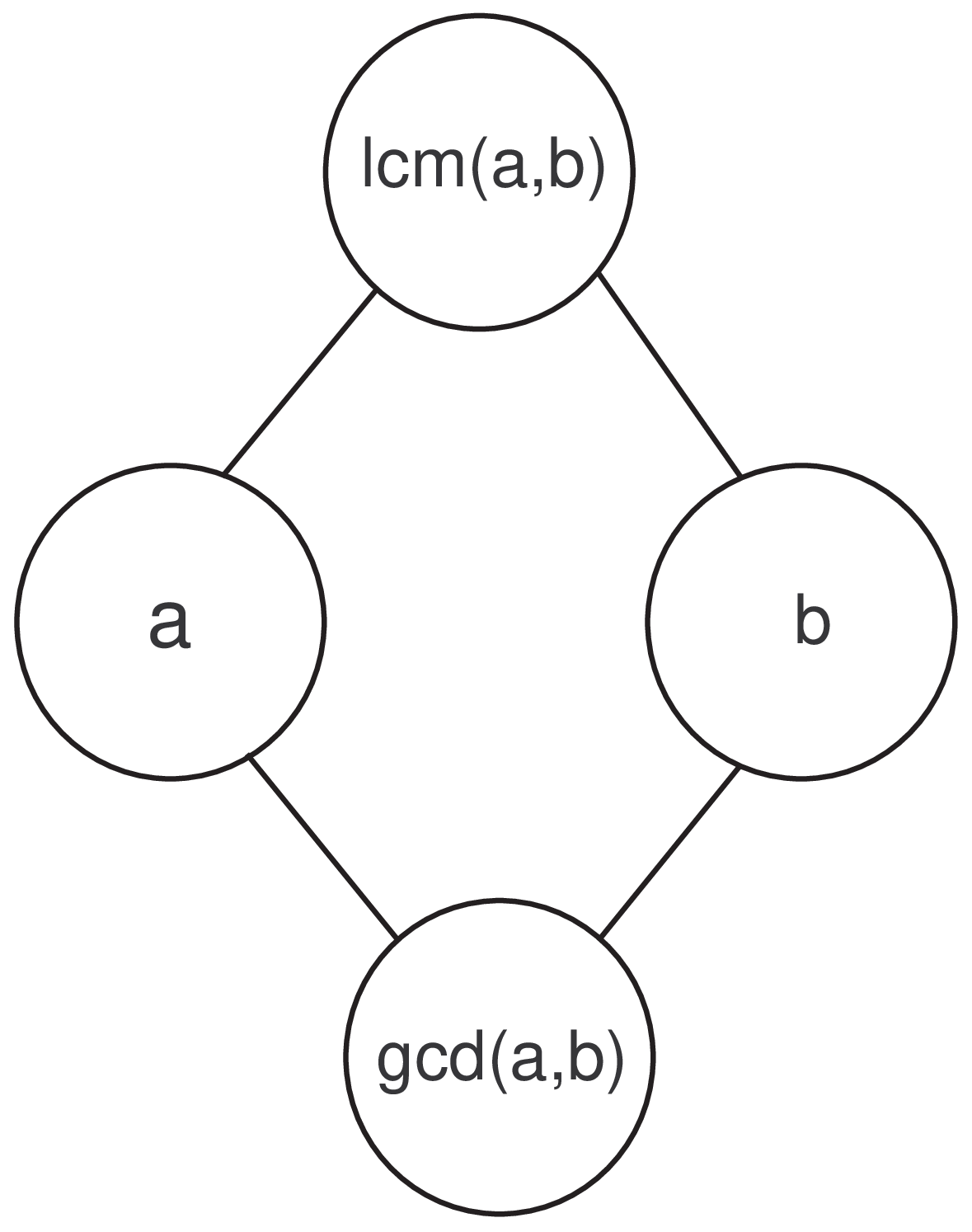}}
\end{center}
\caption{The Hasse diagram of the set $\left\{a,b,\gcd(a,b),\operatorname{lcm}(a,b)\right\}$.}%
\label{1}%
\end{figure}

In the following section, we will define $d(a,b)$ precisely and prove that it
satisfies $1$ and $2$.

\section{Main result}

We begin by reviewing some standard notations.

\begin{definition}
If $n\in\mathbb{N}$ and $p\in\mathbb{P}$, we define, $\nu_{p}(n),$ the
$p$-adic valuation of $n,$ by%
\[
\nu_{p}(n)=\max\left\{  k\in\mathbb{N}_{0}\mathbb{\quad}|\quad p^{k}%
|n\right\}  ,
\]
where $\mathbb{N}_{0}=\mathbb{N\cup}\left\{  0\right\}  .$ It follows from the
fundamental theorem of arithmetic \cite{MR1739433} that%
\[
n=%
{\displaystyle\prod\limits_{p\in\mathbb{P}}}
p^{\nu_{p}(n)}.
\]
We define%
\begin{equation}
\Omega\left(  n\right)  =%
{\displaystyle\sum\limits_{p\in\mathbb{P}}}
\nu_{p}(n). \label{Omega}%
\end{equation}

\end{definition}

The function $\Omega\left(  n\right)  $ is called (surprise!) the \emph{Big
Omega function} \cite[p. 354]{MR568909}. It represents the total number of
prime factors of $n$, counting prime factors with multiplicity. The following
lemma states that $\Omega\left(  n\right)  $ is totally additive.

\begin{lemma}
If $a,b\in\mathbb{N},$ then%
\begin{equation}
\Omega\left(  ab\right)  =\Omega\left(  a\right)  +\Omega\left(  b\right)  .
\label{ab}%
\end{equation}

\end{lemma}

We have now all the necessary elements to define our distance. We denote by
$\mathbb{N}_{0}$ the set $\mathbb{N}\cup\left\{  0\right\}  .$

\begin{definition}
If $a,b\in\mathbb{N},$ we define the function $d:\mathbb{N\times N}%
\rightarrow\mathbb{N}_{0}$ by%
\begin{equation}
d(a,b)=\mathbf{\Omega}\left[  \operatorname{lcm}(a,b)\right]  -\mathbf{\Omega
}\left[  \gcd(a,b)\right]  . \label{d}%
\end{equation}

\end{definition}

Although possible, it is a bit complicated to prove that $d(a,b)$ is a
distance using the definition (\ref{d}). The following theorem gives an
alternative representation for $d(a,b),$ from which it is clear that $d(a,b)$
is indeed a metric.

\begin{theorem}
If $a,b\in\mathbb{N},$ then%
\begin{equation}
d(a,b)=%
{\displaystyle\sum\limits_{p\in\mathbb{P}}}
\left\vert \nu_{p}(a)-\nu_{p}(b)\right\vert . \label{dabs}%
\end{equation}

\end{theorem}

\begin{proof}
Since \cite{MR1739433}
\[
\operatorname{lcm}(a,b)=%
{\displaystyle\prod\limits_{p\in\mathbb{P}}}
p^{\max\left\{  \nu_{p}(a),\nu_{p}(b)\right\}  },\quad\gcd(a,b)=%
{\displaystyle\prod\limits_{p\in\mathbb{P}}}
p^{\min\left\{  \nu_{p}(a),\nu_{p}(b)\right\}  },
\]
then%
\[
d(a,b)=%
{\displaystyle\sum\limits_{p\in\mathbb{P}}}
\left[  \max\left\{  \nu_{p}(a),\nu_{p}(b)\right\}  -\min\left\{  \nu
_{p}(a),\nu_{p}(b)\right\}  \right]  .
\]
But for any real numbers $x,y$%
\[
\max\left\{  x,y\right\}  -\min\left\{  x,y\right\}  =\left\{
\begin{array}
[c]{c}%
x-y,\quad x\geq y\\
y-x,\quad x\leq y
\end{array}
\right.  =\left\vert x-y\right\vert
\]
and the result follows.
\end{proof}

\begin{corollary}
$\left(  \mathbb{N},d\right)  $ is a metric space.
\end{corollary}

Using the metric $d,$ we can give a nice topological interpretation to the set
of prime numbers $\mathbb{P}$.

\begin{example}
If we denote by $\overline{B}_{r}(x)$ the closed ball of radius $r$ centered
at $x,$ i.e.,%
\[
\overline{B}_{r}(x)=\left\{  y\mid d(x,y)\leq r\right\}  ,
\]
we have%
\[
\overline{B}_{1}(1)=\mathbb{P}\text{.}%
\]

\end{example}

Just as the absolute value is a \emph{translation invariant metric}, i.e.,%
\[
\left\vert \left(  x+z\right)  -\left(  y+z\right)  \right\vert =\left\vert
x-y\right\vert ,
\]
the distance $d(a,b)$ is a \emph{multiplicative invariant metric}.

\begin{proposition}
If $a,b,c\in\mathbb{N},$ then%
\begin{equation}
d(ac,bc)=d(a,b). \label{dac}%
\end{equation}

\end{proposition}

\begin{proof}
We have \cite{MR0195783}%
\[
\operatorname{lcm}(ac,bc)=c\operatorname{lcm}(a,b)\text{ \ and \ }%
\gcd(ac,bc)=c\gcd(a,b).
\]
Thus, from (\ref{ab}) we conclude that%
\[
\mathbf{\Omega}\left[  \operatorname{lcm}(ac,bc)\right]  -\mathbf{\Omega
}\left[  \gcd(ac,bc)\right]  =\mathbf{\Omega}\left[  \operatorname{lcm}%
(a,b)\right]  -\mathbf{\Omega}\left[  \gcd(a,b)\right]
\]
and the result follows.
\end{proof}

We should now check that $d$ satisfies the properties (\ref{Prop1}) and
(\ref{Prop2}).

\begin{theorem}
If $a,b\in\mathbb{N}$ and $a<b,$ then%
\[
d(a,b)=1\Leftrightarrow\exists p\in\mathbb{P}\text{ \ such that }b=ap.
\]

\end{theorem}

\begin{proof}
It is clear that
\[
\left\vert \nu_{p}(a)-\nu_{p}(b)\right\vert \in\mathbb{N}_{0},\quad\forall
p\in\mathbb{P}\text{.}%
\]
Thus, from (\ref{dabs}) we have%
\[
1=d(a,b)=%
{\displaystyle\sum\limits_{p\in\mathbb{P}}}
\left\vert \nu_{p}(a)-\nu_{p}(b)\right\vert
\]
if and only if $\exists p\in\mathbb{P}$ such that $\left\vert \nu_{p}%
(a)-\nu_{p}(b)\right\vert =1$ and
\[
\left\vert \nu_{q}(a)-\nu_{q}(b)\right\vert =0\quad\forall q\in\mathbb{P}%
\backslash\left\{  p\right\}  .
\]
Since $a<b,$ we conclude that%
\[
\nu_{p}(b)=\nu_{p}(a)+1\text{ \ and \ }\nu_{q}(a)=\nu_{q}(b)\quad\forall
q\in\mathbb{P}\backslash\left\{  p\right\}
\]
or, equivalently, $b=ap.$
\end{proof}

\begin{theorem}
If $a,b\in\mathbb{N}$, $l=\operatorname{lcm}(a,b)$ and $g=\gcd(a,b),$ then%
\[
d(a,l)+d(l,b)=d(a,b)=d(a,g)+d(g,b).
\]

\end{theorem}

\begin{proof}
We have%
\[
\max\left\{  \nu_{p}(a),\nu_{p}(l)\right\}  =\max\left\{  \nu_{p}%
(a),\max\left\{  \nu_{p}(a),\nu_{p}(b)\right\}  \right\}  =\max\left\{
\nu_{p}(a),\nu_{p}(b)\right\}
\]
and%
\[
\min\left\{  \nu_{p}(a),\nu_{p}(l)\right\}  =\min\left\{  \nu_{p}%
(a),\max\left\{  \nu_{p}(a),\nu_{p}(b)\right\}  \right\}  =\nu_{p}(a).
\]
Hence,%
\begin{equation}
d(a,l)+d(l,b)=2\mathbf{\Omega}\left[  \operatorname{lcm}(a,b)\right]
-\Omega(a)-\Omega(b).\label{1}%
\end{equation}
Using \cite{MR0195783}%
\begin{equation}
ab=\operatorname{lcm}(a,b)\gcd(a,b)\label{lcmgdc}%
\end{equation}
and (\ref{ab}) in (\ref{1}), we obtain%
\[
d(a,l)+d(l,b)=\mathbf{\Omega}\left[  \operatorname{lcm}(a,b)\right]
-\mathbf{\Omega}\left[  \gcd(a,b)\right]  =d(a,b).
\]

Since%
\[
\max\left\{  \nu_{p}(a),\nu_{p}(g)\right\}  =\max\left\{  \nu_{p}%
(a),\min\left\{  \nu_{p}(a),\nu_{p}(b)\right\}  \right\}  =\nu_{p}(a)
\]
and%
\[
\min\left\{  \nu_{p}(a),\nu_{p}(g)\right\}  =\min\left\{  \nu_{p}%
(a),\min\left\{  \nu_{p}(a),\nu_{p}(b)\right\}  \right\}  =\min\left\{
\nu_{p}(a),\nu_{p}(b)\right\}  ,
\]
then%
\begin{equation}
d(a,g)+d(g,b)=\Omega(a)+\Omega(b)-2\mathbf{\Omega}\left[  \gcd(a,b)\right]
.\label{2}%
\end{equation}
Using (\ref{ab}) and (\ref{lcmgdc}) in (\ref{1}), the result follows.
\end{proof}

The number of elements in the set $S_{k}=\left\{  m\in\mathbb{N}\mid
\Omega(m)=k\right\}  $ is clearly infinite. A more interesting question would
be to describe the number of elements in the set $S_{k}\cap I_{n}$ as
$n\rightarrow\infty.$ We have \cite[22.18]{MR568909}
\begin{equation}
\#\left(  S_{k}\cap I_{n}\right)  \sim\frac{n}{\ln(n)}\frac{\left[  \ln
\ln(n)\right]  ^{k-1}}{\left(  k-1\right)  !},\quad n\rightarrow
\infty,\label{num theo}%
\end{equation}
where $\#$ represents cardinality and $a\sim b$ means that $\frac{a}%
{b}\rightarrow1$ as $n\rightarrow\infty.$ The proof of (\ref{num theo}) is
beyond the reach of this paper, since it contains (or depends on) a proof of
the Prime Number Theorem \cite{MR1994094}. 

What we can do instead is  estimate the maximum distance between two numbers
in the set $I_{n},$ which we will do in the next section.

\subsection{The diameter of $I_{n}$}

\begin{definition}
If $s\in\left(  0,\infty\right)  ,\ p\in\left(  1,\infty\right)  ,$ let%
\begin{equation}
\xi_{p}\left(  s\right)  =\max\left\{  k\in\mathbb{N}\quad|\quad p^{k}\leq
s\right\}  .\label{xi}%
\end{equation}

\end{definition}

The next couple of lemmas follow immediately from the definition of $\xi
_{p}\left(  s\right)  .$

\begin{lemma}
If $s\in\left(  0,\infty\right)  ,\ p\in\left(  1,\infty\right)  ,$ then%
\[
\xi_{p}\left(  s\right)  =\left\lfloor \frac{\ln(s)}{\ln(p)}\right\rfloor ,
\]
where%
\[
\left\lfloor x\right\rfloor =\max\left\{  k\in\mathbb{Z}\quad|\quad k\leq
x\right\}  .
\]

\end{lemma}

\begin{lemma}
If $m\in\left(  0,\infty\right)  ,\ p\in\left(  1,\infty\right)  ,$ then

\begin{enumerate}
\item
\begin{equation}
\xi_{q}\left(  m\right)  \leq\xi_{p}\left(  m\right)  ,\qquad\text{if \ }p\leq
q.\label{pq}%
\end{equation}

\item
\begin{equation}
\xi_{p}\left(  m\right)  \leq\xi_{p}\left(  n\right)  \qquad\text{if \ }m\leq
n.\label{mn}%
\end{equation}

\end{enumerate}
\end{lemma}

We can now obtain a first estimate comparing the growth of $\Omega\left(
n\right)  $ and $\xi_{p}\left(  n\right)  .$

\begin{lemma}
Let $n$ $\in\mathbb{N}.$ Then,

\begin{enumerate}
\item For all $n$ $\in\mathbb{N}$
\begin{equation}
\Omega\left(  n\right)  \leq\xi_{2}\left(  n\right)  .\label{even}%
\end{equation}

\item If $n$ $\in\mathbb{N}$ is odd, then%
\begin{equation}
\Omega\left(  n\right)  \leq\xi_{3}\left(  n\right)  .\label{odd}%
\end{equation}

\end{enumerate}
\end{lemma}

\begin{proof}

\begin{enumerate}
\item Since%
\[
2^{\Omega(n)}=%
{\displaystyle\prod\limits_{p\in\mathbb{P}}}
2^{\nu_{p}(n)}\leq%
{\displaystyle\prod\limits_{p\in\mathbb{P}}}
p^{\nu_{p}(n)}=n,
\]
the result follows from (\ref{xi}).

\item Similarly, if $n$ is odd, then%
\[
3^{\Omega(n)}=%
{\displaystyle\prod\limits_{p\in\mathbb{P}}}
3^{\nu_{p}(n)}\leq%
{\displaystyle\prod\limits_{p\in\mathbb{P}}}
p^{\nu_{p}(n)}=n,
\]
since $\nu_{2}(n)=0.$
\end{enumerate}
\end{proof}

We have now all the necessary elements to prove our result on the diameter of
$I_{n}.$

\begin{theorem}
Let $n$ $\in\mathbb{N}.$ Then,
\[
\delta\left(  I_{n}\right)  =\xi_{2}\left(  n\right)  +\xi_{3}\left(
n\right)  ,
\]
where%
\[
\delta(A)=\sup\left\{  d(x,y)\mid x,y\in A\right\}
\]
is the diameter of $A$.
\end{theorem}

\begin{proof}
Let $x,y\in I_{n}.$ Then,
\[
d(x,y)\leq d(x,1)+d(1,y)=\Omega\left(  x\right)  +\Omega\left(  y\right)  .
\]
We have three possibilities:

(a) If $x$ and $y$ are odd numbers then, from (\ref{pq}), (\ref{mn}) and
(\ref{odd}) we have%
\[
\Omega\left(  x\right)  +\Omega\left(  y\right)  \leq\xi_{3}\left(  x\right)
+\xi_{3}\left(  y\right)  \leq\xi_{2}\left(  n\right)  +\xi_{3}\left(
n\right)  .
\]

(b) If $x$ or $y$ is an odd number then, from (\ref{even}), (\ref{mn}) and
(\ref{odd}) we get%
\begin{equation}
d(x,y)\leq\Omega\left(  x\right)  +\Omega\left(  y\right)  \leq\xi_{2}\left(
x\right)  +\xi_{3}\left(  y\right)  \leq\xi_{2}\left(  n\right)  +\xi
_{3}\left(  n\right)  .\label{(b)}%
\end{equation}

(c) If $x$ and $y$ are even numbers, let
\[
g=\gcd(x,y)\text{ \ and \ }x=ag,\quad y=bg.\text{\ }%
\]
Then $a,b\in I_{n}$ and $a$ or $b$ is an odd number. Using (\ref{dac}),
(\ref{mn}) and (\ref{(b)}) we obtain
\[
d(x,y)=d(a,b)\leq\xi_{2}\left(  n\right)  +\xi_{3}\left(  n\right)  .
\]
Hence, we conclude that%
\[
\delta\left(  I_{n}\right)  \leq\xi_{2}\left(  n\right)  +\xi_{3}\left(
n\right)  .
\]

On the other hand, letting
\[
x=2^{\xi_{2}\left(  n\right)  },\quad y=3^{\xi_{3}\left(  n\right)  }%
\]
we have $x,y\in I_{n}$ and therefore%
\[
\xi_{2}\left(  n\right)  +\xi_{3}\left(  n\right)  =d(x,y)\leq\delta\left(
I_{n}\right)  .
\]

\end{proof}

In the next section, we will extend the definition of $d$ to a bigger subset
of the real numbers, which contains the rational numbers.

\section{Extension}

We remind the reader that $l_{1}$ is the space of absolutely summable
sequences, i.e.,%
\[
l_{1}=\left\{  \left(  b_{k}\right)  _{k=1}^{\infty}\mid\left\Vert \left(
b_{k}\right)  _{k=1}^{\infty}\right\Vert _{1}<\infty\right\}  ,
\]
where $\left(  b_{k}\right)  _{k=1}^{\infty}$ represents the sequence
$b_{1},b_{2},\ldots,$ and the norm $\left\Vert \cdot\right\Vert _{1}$ on
$l_{1}$ is defined by%
\[
\left\Vert \left(  b_{k}\right)  _{k=1}^{\infty}\right\Vert _{1}=%
{\displaystyle\sum\limits_{k=1}^{\infty}}
\left\vert b_{k}\right\vert .
\]

\begin{definition}
Let $\mathbb{M}$ be defined by%
\[
\mathbb{M}=\left\{  x\in\mathbb{R}^{+}\mid x=%
{\displaystyle\prod\limits_{k=1}^{\infty}}
p_{k}^{\alpha_{k}}\text{ \ and }\left(  \alpha_{k}\ln k\right)  _{k=1}%
^{\infty}\in l_{1}\right\}  ,
\]
where $\mathbb{P}=\left\{  p_{1},p_{2,}p_{3},\ldots\right\}  .$

With the notation above, we define $\nu_{p_{k}}(x)=\alpha_{k}$ for
$x\in\mathbb{M}$ and%
\[
\Omega\left(  x\right)  =%
{\displaystyle\sum\limits_{k=1}^{\infty}}
\nu_{p_{k}}(x).
\]

\end{definition}

\begin{remark}
\begin{enumerate}
\item Clearly $\mathbb{Q}^{+}\subset\mathbb{M},$ but also irrational numbers
like $\sqrt[n]{a},$for $a\in\mathbb{N}.$

\item The condition $\left(  \alpha_{k}\ln k\right)  _{k=1}^{\infty}\in l_{1}$
warranties the existence of the infinite product, since from the Prime Number
Theorem \cite{MR1994094} we have $p_{k}\sim k\ln(k)$ as $k\rightarrow\infty.$

\item If $x\in\mathbb{M},$ we get%
\[
\Omega\left(  x\right)  =%
{\displaystyle\sum\limits_{k=1}^{\infty}}
\alpha_{k}<%
{\displaystyle\sum\limits_{k=1}^{\infty}}
\left\vert \alpha_{k}\right\vert \ln\left(  k\right)  <\infty.
\]

\end{enumerate}
\end{remark}

We can now extend our definition (\ref{d}).

\begin{definition}
Let $x,y\in\mathbb{M}.$ We define the distance $d(x,y)$ by%
\begin{equation}
d(x,y)=%
{\displaystyle\sum\limits_{p\in\mathbb{P}}}
\left\vert \nu_{p}(x)-\nu_{p}(y)\right\vert . \label{d1}%
\end{equation}

\end{definition}

\begin{remark}
If we define the function $\Psi:\mathbb{M}\rightarrow l_{1}$ by%
\[
\Psi(x)=\left(  \nu_{p_{k}}(x)\right)  _{k=1}^{\infty},
\]
then from (\ref{d1}) we see that $\Psi$ is an isometry \cite{MR0085462}
between the metric spaces $\left(  \mathbb{M},d\right)  $ and $\left(
l_{1},\left\Vert \cdot\right\Vert _{1}\right)  .$
\end{remark}

\end{document}